\documentclass[preprint]{elsarticle}

\usepackage[latin1]{inputenc} 
\usepackage[T1]{fontenc}
\usepackage[english]{babel}
\usepackage[dvipsnames]{xcolor} 
\usepackage{framed} 
\colorlet{shadecolor}{red!20}
\usepackage{indentfirst} 
\usepackage{latexsym}
\usepackage{amssymb}
\usepackage{amsmath}
\usepackage{amsthm}
\usepackage{url}

\usepackage{enumitem}
\setlist[description]{leftmargin=\parindent,labelindent=\parindent}

\newtheorem{thm}{Theorem}[section]
\newtheorem{cor}[thm]{Corollary}
\newtheorem{lem}[thm]{Lemma}
\newtheorem{prop}[thm]{Proposition}
\newtheorem{defn}[thm]{Definition}

\newtheorem{question}[thm]{Question}
\newtheorem{oss}{Observation}[section]
\newtheorem{exam}{Example}[section]

\newtheorem*{thm*}{Theorem}

\newcommand{\U}{\mathcal{U}}
\newcommand{\N}{\mathbb{N}}
\newcommand{\Z}{\mathbb{Z}}

\newcommand{\R}{\mathbb{R}}
\newcommand{\V}{\mathcal{V}}

\newcommand{\bN}{\beta\mathbb{N}}
\newcommand{\bZ}{\beta\mathbb{Z}}
\newcommand{\FS}{\operatorname{FS}}
\newcommand{\FP}{\operatorname{FP}}
\newcommand{\Pfin}{\wp_{\operatorname{fin}} }
\newcommand{\MT}{\operatorname{MT}}

\renewcommand{\cite}{\citep}

\numberwithin{equation}{section}

\begin{document}

\title{Rado equations solved by linear combinations of idempotent ultrafilters}

\author{Lorenzo Luperi Baglini %
\fnref{fn1}}
\address{Dipartimento di Matematica, Universit\`{a} di Milano, Via Saldini 50, 20133 Milano, Italy.}
\ead{lorenzo.luperi@unimi.it}
\cortext[cor1]{Corresponding author}

\author{Paulo Henrique Arruda\corref{cor1}\fnref{fn2}}
\address{Universit\"{a}t Wien, Fakult\"{a}t f\"{u}r Mathematik, Oskar-Morgenstern-Platz 1, 1090 Vienna, Austria.}
\ead{paulo.arruda@univie.ac.at}

\fntext[fn1,fn2]{Both supported by grant P30821-N35 of the Austrian Science Fund FWF.}

\begin{keyword}
Partition regularity\sep ultrafilters\sep Rado equations. \MSC[2020]{Primary 5D10, 11D04; Secondary 54D35, 54D80.}
\end{keyword}

\date{}

\begin{abstract}
We fully characterise the solvability of Rado equations inside linear combinations $a_{1}\U+\dots+ a_{n}\U$ of idempotent ultrafilters $\U\in\beta\Z$ by exploiting known relations between such combinations and strings of integers. This generalises a partial characterisation obtained previously by Mauro Di Nasso.
\end{abstract}

\maketitle
\section{Introduction}

A long studied problem in combinatorics deals with the partition regularity of Diophantine equations.

\begin{defn} Let $S$ be either $\N$ or $\Z$, let $m\in\N$ and let $P\in\Z\left[x_{1},\dots,x_{m}\right]$. We say that the equation $P\left(x_1,\dots,x_m\right)=0$ is partition regular on $S$ if for every finite partition $C_1,\dots,C_r$ of $S$  one can find an $i\in\{1,\dots,r\}$ and $a_1,\dots,a_m\in C_i\setminus\{0\}$ such that $P\left(a_1,\dots,a_m\right)=0$.
\end{defn}

The earliest known result about the partition regularity of equations is due to I. Schur (see \cite{schur}) who established the partition regularity of the equation $x+y=z$ on $\N$. Later, Rado generalised Schur's Theorem and provided a necessary and sufficient condition for a finite system of linear homogeneous Diophantine equations to be partition regular \cite{Rado1933}. For a single linear homogeneous equation, Rado's result reads as follows:

\begin{thm} (Rado's Theorem) Let $S$ be either $\N$ or $\Z$. A linear homogeneous equation $c_1x_1+\dots + c_mx_m = 0$, with integer coefficients, is partition regular on $S$ if and only if there is a non-empty $I\subseteq\{1,\dots,m\}$ such that $\sum_{i\in I}c_i = 0$.
\end{thm}

Motivated by Rado's Theorem, we introduce the following definition.
\begin{defn}
A linear homogeneous polynomial $P(x_1,\dots,x_m) = \sum_{i=1}^{m} c_ix_i$, such that $c_i\in\Z\setminus\{0\}$, is said to be a Rado polynomial if there is a non empty $I\subseteq\{1,\dots,m\}$ such that $\sum_{i\in I}c_i = 0$.
If $P$ is Rado, we will also say that the equation $P(x_1,\dots,x_m)=0$ is a Rado equation.
\end{defn}

Following the works of F. Galvin and S. Glazer (see \cite[Section 5.6]{HindmanStrauss2011} for historical remarks), ultrafilters on semigroups have been one of the tools used for the study of partition regularity of equations. We identify the set of all ultrafilters on a set $S$ to $\beta S$, namely the \v{C}ech-Stone compactification of $S$ as a discrete space. The general relationship between partition regularity and ultrafilters, that can be made even more precise in an abstract setting (see e.g. \cite[Theorem 5.7]{HindmanStrauss2011} or \cite[Proposition 1.8]{LuperiBagliniDiNasso2018}), can be formulated for Diophantine equations as follows:

\begin{thm}\label{sempliciotto} Let $S$ be either $\N$ or $\Z$, let $m\in\N$, and let $P\in\Z\left[x_{1},\dots,x_{m}\right]$ be given. The equation $P\left(x_{1},\dots,x_{m}\right)=0$ is partition regular on $S$ if and only if there is an ultrafilter $\U\in\beta S$ such that for every $A\in\U$, one can find $a_1,\dots,a_m\in A$ satisfying $P\left(a_1,\dots,a_m\right)=0$.
\end{thm}

Motivated by the above result, we introduce a notation that we will often use in what follows.

\begin{defn}
	Let $S$ be either $\N$ or $\Z$, let $m\in\N$ and $P\in\Z\left[x_{1},\dots,x_{m}\right]$. We say that an ultrafilter $\U\in\beta S$ witnesses the partition regularity of the equation $P\left(x_{1},\dots,x_{m}\right)=0$ if for all $A\in \U$ there exist $a_1,\dots,a_m\in A$ satisfying $P\left(a_{1},\dots,a_{m}\right)=0$. In this case, we write $\U\models P\left(x_1,\dots,x_m\right)=0$.
\end{defn}

In recent years, the "qualitative" problem of finding which classes of ultrafilters witness the partition regularity of a given equation has become important due to some new techniques that allow building new partition regular equations from equations whose partition regularity is witnessed by a common ultrafilter. This kind of idea has been used, e.g., in \cite{BeiglbockBergelsonDownarowiczFish2009, DiNasso2014, LuperiBagliniDiNasso2018, LuperiBaglini2021,LuperiBaglini2013}. For instance, in \cite{LuperiBaglini2013} (see also \cite{LuperiBagliniDiNasso2018}), it was shown that any ultrafilter that witnesses the partition regularity of both $x+y=z$ and $uv=t$ will also witness the partition regularity of $x+y=uv$ among many others. 

In particular, in \cite{DiNasso2014}, Di Nasso, using a nonstandard framework and considerations about strings of integers close to those that we will use in Section \ref{section:Polynomials_solved_by_lin_comb_idemp}, has proven that certain linear combinations $a_{1}\U+\dots+ a_{n}\U$, with $a_{i}\in\N$, $\U\in\beta\N$, witness the partition regularity of a class of Rado equations.

\begin{thm}\cite[Theorem 1.2]{DiNasso2014}\label{theorem:DiNasso_linear_thm} Let $m>2$. For every $c_1,\dots,c_m\in\Z$ satisfying $c_1+\dots+c_m=0$, there are $a_1,\dots,a_{m-2}\in\N$ such that, for every additively idempotent ultrafilter $\U\in\bN$, the ultrafilter $a_1\U+\dots+ a_{m-2}\U$
witnesses the partition regularity of the equation $c_1x_1+\dots + c_mx_m=0$.
\end{thm}

The above result is the major inspiration for this work; in fact, in this paper, we extend Di Nasso's result by finding a complete characterisation of which Rado equations are solved by linear combinations of the form $a_{1}\U+\dots+ a_{n}\U$ for $a_{1},\dots,a_{n}\in \Z$ and $\U\in\beta\Z$ an additively idempotent ultrafilter; hence dropping Di Nasso's assumptions that $c_{1}+\dots+c_{m}=0$ and $\U\in\bN$. This characterisation will use the relationship between the linear combinations $a_{1}\U+\dots+ a_{n}\U$ and certain sets of strings of integers, i.e. $k$-tuples of integers, for $k\geq 1$. As an interesting consequence of our main Theorem \ref{lingen}, see Corollary \ref{shursequation}, we prove that, up to multiplication by a scalar, the Schur's Equation $x+y=z$ is the only linear homogeneous equation in three variables satisfying the following:
\begin{itemize}
    \item there exists $k\in\N$ and $a_1,\dots,a_k$ such that for all additively idempotent ultrafilter $\U\in\bZ$, $a_1\U+\dots+ a_k\U$ solves the equation $x+y=z$.
\end{itemize}
Notice that, as we will show, for all $k\in\N$, $a_1,\dots,a_k\in\Z\setminus\{0\}$ and  additively idempotent ultrafilter $\U\in\bZ$, $a_1\U+\dots+ a_k\U$ solves the Schur's Equation.

The paper is organised as follows: in Section \ref{ultra algebra} we recall all the basic results about idempotent ultrafilters that we will need, particularly the definition of a strongly summable ultrafilters on semigroups; in Section \ref{strings} we will talk about the relationships between strings of integers, linear combinations of idempotent ultrafilters and solvability of equations by strings of integers; in Section \ref{section:Polynomials_solved_by_lin_comb_idemp} we will present our main result, namely the characterisation of which Rado equations are solved by which kind of linear combinations of idempotents. Our result uses the existence of strongly summable ultrafilters, which is independent from ZFC. In Section \ref{FI} we will show how a minor modification of our main result can be obtained in ZFC. Finally, in Section \ref{exampleeee} we show some explicit examples, proving also that there are Rado equations $P\left(x_{1},\dots,x_{m}\right)=0$ such that for all $\left(a_{1},\dots,a_{n}\right)$ there are idempotents $\U\in\beta\Z$ so that $a_{1}\U+\dots+ a_{n}\U$ does not witness the partition regularity of $P\left(x_{1},\dots,x_{m}\right)=0$.

\section{Ultrafilters and their algebra}\label{ultra algebra}

We assume that the reader knows the basic notions of the algebra on the \v{C}ech-Stone compactification of semigroups; in this Section, we only recall some known facts about idempotent ultrafilters and their connection with the partition regularity of equations. As it is well known, idempotent ultrafilters are related with the concept of finite products on semigroups: if $(S,\cdot)$ is a semigroup, let $\Pfin(S)$ denotes the collection of all non-empty finite subsets of $S$; given a sequence $(x_n)_{n\in\N}$ of elements of $S$ and a non-empty finite $F\subseteq\N$ enumerated as $n_1<\dots<n_k$, let 
\begin{equation*}
    \prod_{n\in F} x_n = x_{n_1}\cdot \dots\cdot x_{n_k}.
\end{equation*}
Define the set of all finite products of $(x_n)_{n\in\N}$ as the set
\begin{equation*}
	\FP\big((x_n)_{n\in\N}\big) = \left\{\prod_{n\in F} x_n\mid F\in\Pfin(\N)\right\}.
\end{equation*}
If $S$ is an Abelian semigroup, e.g. $(\N,+)$ or $(\Z,+)$, and the additive notation is adopted instead, this notion is translated to \emph{finite sums} as defined by 
\begin{equation*}
	\FS\big((x_n)_{n\in\N}\big) = \left\{\sum_{n\in F} x_n \mid  F\in\Pfin(\N)\right\}.
\end{equation*}
A set $A\subseteq S$ that contains the finite products of some injective sequence is called an \emph{IP-set}. It is well known that a subset $A$ of $S$ contains $\FP\left((x_{n})_{n\in\N}\right)$ for some injective sequence $\left(x_{n}\right)_{n\in\N}$ if and only if $A$ belongs to some idempotent ultrafilter $\U\in\beta S$ (see e.g. \cite[Theorem 5.12 ]{HindmanStrauss2011}). As a consequence of this fact, one gets the famous Hindman's Theorem\footnote{F. Galvin and S. Glazer were the first to give a proof of Hindman's Theorem, in the case $(S,\cdot)=(\mathbb{N},+)$, using idempotent ultrafilters.}:

\begin{thm} (Hindman's Theorem) For every finite partition $C_{1},\dots, C_{r}$ of a semigroup $S$ there exists an injective sequence $(x_n)_{n\in\N}$ and $i\in\{1,\dots,r\}$ such that 
\begin{equation*}
\FP((x_n)_{n\in\N})\subseteq C_i.
\end{equation*}
\end{thm}
However, in general, given an idempotent ultrafilter $\U\in\beta S$, the set of finite products of some injective sequence itself will not belong to $\U$, even in the cases that $S$ is either $(\N,+)$ or $(\Z,+)$. To have this stronger property, we have to consider a special class of ultrafilters (see \cite[Chapter 12]{HindmanStrauss2011}), namely \emph{strongly summable ultrafilters}:

\begin{defn}\label{definition:ss-ultra-N}
	An ultrafilter $\U$ on a semigroup $S$ is said to be a strongly productive ultrafilter if for every $A\in\U$ one can find an injective sequence $(x_n)_{n\in\N}$ of elements of $S$ such that $\FP\big((x_n)_{n\in\N}\big)\subseteq A$ and $ \FP\big((x_n)_{n\in\N}\big)\in \U$. If $S$ is Abelian and the additive notation is adopted, we call $\U$ a strongly summable ultrafilter instead.
\end{defn}

One can easily deduce that, if $\U\in\bZ$ is a strongly summable ultrafilter, either $\U$ or $-\U$ is a strongly summable ultrafilter\footnote{Given a $\U\in\bZ$, we denote by $-\U$ the ultrafilter $\U_{-1}\cdot \U$, where $\U_{-1}$ is the principal ultrafilter generated by $-1$; alternatively, $A\in -\U$ iff $\{k\in\Z:-k\in A\}\in\U$. Moreover, given a $\V\in\bZ$, we write $\V-\U$ as an abbreviation for $\V+(-\U)$.} in $\bN$.

\begin{thm}\cite[Theorem 2.3]{HindmanProstatovStrauss}, \cite[Lemma 2.2]{HindmanLeggetteJones}\label{theorem:ss-ultra-idempotent} If $S$ is a free semigroup or a countable Abelian group, then every strongly productive ultrafilter on $S$ is idempotent.
\end{thm}

As a consequence of Theorem \ref{theorem:ss-ultra-idempotent}, every strongly summable ultrafilter on $(\N,+)$ or $(\Z,+)$ is idempotent. In \cite[Theorem 1.1]{BretonLupini}, the authors provide stronger generalisations of Theorem \ref{theorem:ss-ultra-idempotent}; nevertheless, as observed in \cite[Conjecture 1.6]{BretonLupini}, it is unknown, at the present time, if Theorem \ref{theorem:ss-ultra-idempotent} holds for any semigroup.

A. Blass and N. Hindman proved that the existence of a strongly summable ultrafilters on $\N$ implies the existence of P-points on $\bN$ \cite[Sections 12.3 and 12.5]{HindmanStrauss2011}; as proved by S. Shelah \cite{Wimmers}, the existence of P-points on $\bN$ cannot be obtained within ZFC, thus the existence of strongly summable ultrafilters is independent from ZFC. It is attributed to E. van Douwen \cite[Sections 12.2 and 12.5]{HindmanStrauss2011} that if either the Continuum Hypothesis or the Martin's Axiom hold, then  there is a strongly summable ultrafilter on $\N$, hence the existence of strongly summable ultrafilters\footnote{Notice that several of the mentioned proofs were done for \emph{union ultrafilters} \cite[Definition 12.30]{HindmanStrauss2011}. Since the notions of union ultrafilters and strongly summable ultrafilters coincide on $\N$, one can easily derive the existence of the latters from the existence of the formers.} is consistent with ZFC. More recently, in \cite{Eisworth}, T. Eisworth proved that $\operatorname{cov}(\mathcal{M})=\mathfrak{c}$  implies\footnote{If $\mathcal{M}$ is the ideal of meagre sets of $\R$, the cardinal $\operatorname{cov}(\mathcal{M})$ is the covering of $\mathcal{M}$, defined as $\min\{|\mathcal{C}|:\mathcal{C}\subseteq\mathcal{M}\text{ and }\bigcup\mathcal{C}=\R\}$. The equality $\operatorname{cov}(\mathcal{M})=\mathfrak{c}$ is equivalent to the assertion that Martin's Axiom holds for countable partial orders, which is independent from ZFC. The assertion $\operatorname{cov}(\mathcal{M})<\mathfrak{c}$ is equivalent to the fact that Martin's Axiom fails for countable partial orders \cite[Section 1.1]{Breton} and it is also independent from ZFC.} the existence of strongly summable ultrafilters on $\N$  (see also \cite[Theorem 2.8]{Breton}), thus weaking the assuption of the Martin's Axiom. It is also known that  $\operatorname{cov}(\mathcal{M})<\mathfrak{c}$ is consistent with the existence of strongly summable ultrafilters on Abelian groups \cite[Section 4.3]{Breton}.

Theorem \ref{theorem:ss-ultra-idempotent} and the Lemma \ref{lemma:ss_ultrafilter_main_lemma} below are the only technical results about strongly summable ultrafilters that we will use. The proof of the Lemma \ref{lemma:ss_ultrafilter_main_lemma} is almost identical to the proof of \cite[Lemma 12.20 ]{HindmanStrauss2011}, which originally was proved for $\U\in\bN$ and $k=4$. 

\begin{lem}\label{lemma:ss_ultrafilter_main_lemma}
	Let $k\geq 4$, $\U\in\beta\Z$ be a strongly summable ultrafilter and $A\in\U$. Then there is an injective sequence $(x_n)_{n\in\N}$ in $\Z$ such that
	\begin{enumerate}
		\item $\FS\big((x_n)_{n\in\N}\big)\subseteq A$;
		\item for each $m\in\N$, $\FS\big((x_n)_{n\geq m}\big)\in\U$;
		\item for each $n\in\N$, $|x_{n+1}| > k|\sum_{t=1}^{n}x_t|$.
	\end{enumerate}
\end{lem}

Alongside other techniques, ultrafilters and their algebra have been useful in the study of partition regularity of equations and systems of equations since Galvin and Glazer's proof of Hindman Theorem; for classical results such as Schur's Theorem, van der Waerden's Theorem, Hindman's Theorem, the Central Set theorem, partition regularity of matrices, the Milliken-Taylor Theorem and Hales-Jewett Theorem, among others, see the monograph \cite{HindmanStrauss2011}. More recently, results for nonlinear equations have been proven also using ultrafilters in \cite{DiNasso2014,LuperiBagliniDiNasso2018,LuperiBaglini2021,LuperiBaglini2013}.

 Although the definition of partition regularity can be stated more generally, we restrict ourselves here to the partition regularity of linear Diophantine equations on $\N$ or $\Z$. As recalled in the introduction, the partition regularity of such equations has been characterised by Rado in \cite{Rado1933}. When $P\in\Z[x_1,\dots,x_m]$ is a Rado polynomial, the set of ultrafilters witnessing the partition regularity of the equation $P\left(x_{1},\dots,x_{m}\right)=0$ has several good algebraic properties. We summarise known facts and their respective references in the Theorem below, in which $K(\beta S)$ denotes the minimal bilateral ideal of $\beta S$.

\begin{thm}\label{theorem:facts_about_witnessing_pr}
	Let $S$ be either $\N$ or $\Z$, let $m\in\N$ and let $P\in\Z\left[x_{1},\dots,x_{m}\right]$ be a linear homogeneous polynomial such that $P\left(x_{1},\dots,x_{m}\right)=0$ is partition regular on $S$. Let $\U,\V\in\beta S$. Then
\begin{enumerate}
\item[(i)] the set $	\mathcal{I}_P = \{\U\in\beta S\mid  \U\models P\left(x_1,\dots,x_n\right)=0\}$ is a closed multiplicative bilateral ideal of $\beta S$;
	\item[(ii)] if $\U\models P\left(x_{1},\dots,x_{m}\right)=0$ and $\V\models P\left(x_{1},\dots,x_{m}\right)=0$, then for each $a,b\in\Z$, $a\U+ b\V\models P\left(x_{1},\dots,x_{m}\right)=0$;

	\item[(iii)] if $\U\models P\left(x_{1},\dots,x_{m}\right)=0$, then $ \U\cdot\V\models P\left(x_{1},\dots,x_{m}\right)=0$ and $ \V\cdot\U\models P\left(x_{1},\dots,x_{m}\right)=0$;

	\item[(iv)] if $\U\in\beta S$ is an additively minimal idempotent ultrafilter and $P$ is a Rado polynomial, then $\U\models P\left(x_{1},\dots,x_{m}\right)=0$;

	\item[(v)] if $\U\in \overline{K(\beta\Z,\cdot)}$, then $\U\models P\left(x_{1},\dots,x_{m}\right)=0$.
\end{enumerate}
\end{thm}

\begin{proof}
The proof of (i) can be found in the proof of Proposition 1.8 of \cite{LuperiBagliniDiNasso2018}. 

The proof of (ii) can be done using Theorem 3 of \cite{LuperiBaglini2013} and the fact that we are dealing with a linear homogeneous polynomial. 

The truth of (iii) can easily derived from (i): as $\mathcal{I}_P$ is a multiplicative bilateral ideal of $\beta S$ and $\U\in\mathcal{I}_P$ by hypothesis, for any $\V\in\beta S$, $U\cdot \V$ and $\V\cdot \U$ are both elements of $\mathcal{I}_P$. 

The proof of (iv) is a consequence of Theorem 2 of \cite{BeiglbockBergelsonDownarowiczFish2009}, in which the authors proved that any additively minimal idempotent element\footnote{If $S$ is either $\N$ or $\Z$, an additively idempotent ultrafilter $\U\in\beta S$ is essential if every $A\in\U$ has positive Banach density. In \cite[Theorem 2]{BeiglbockBergelsonDownarowiczFish2009}, the authors proved that any essential idempotent ultrafilter witnesses the partition regularity of any Rado system. As observed in the paragraph before \cite[Theorem 1.14]{BergelsonMcCutcheon2010}, every additively minimal idempotent ultrafilter is also an essential idempotent.} of $\beta S$ witnesses the partition regularity of any Rado equation.

To prove (v), let us observe that, as $\mathcal{I}_P$ is a multiplicative bilateral ideal of $\beta S$ and by the definition of the minimal bilateral ideal, $K(\beta S,\cdot)\subseteq\mathcal{I}_P$ and the fact that $\mathcal{I}_P$ is closed ensures that $\overline{K(\beta S,\cdot)}\subseteq \mathcal{I}_P$. As such, any element of $\overline{K(\beta S,\cdot)}$ witnesses the partition regularity of any Rado equation.
\end{proof}

Although we will use the above results only in the linear case, we recall here that Theorem \ref{theorem:facts_about_witnessing_pr} displayed a major importance in the study of the partition regularity of nonlinear polynomials in \cite{DiNasso2017, LuperiBagliniDiNasso2018,LuperiBaglini2021,LuperiBaglini2013}.

\section{String solutions to linear polynomials}\label{strings}

In this section we introduce basic results and notations about the relationships between strings of integers, ultrafilters and partition regularity problems. The relationship is well known in the context of Milliken-Taylor systems, as recalled in the Introduction. Here, however, we are interested in the relationships with single linear equations. We recall some basic facts from \cite{DiNasso2017}; however, whilst in that paper linear combinations of ultrafilters where studied from the point of view of nonstandard analysis, we will not use a nonstandard approach here. We denote by $\N_0$ the set of all non-negative integers, i.e. $\N_0 = \N\cup\{0\}$. In this section, consider $S$ to be either $\N_0$ or $\Z$.

 \begin{defn}
 	For each integer $k\geq 0$, let $S^k$ be the set of all $k$-tuples or $k$-strings of elements of $S$ (a $0$-tuple is the empty string). Denote by $S^{<\omega}$ the set of all finite strings of elements of $S$, i.e. $	S^{<\omega} = \bigcup_{k\geq 0} S^k.$
 \end{defn}

\begin{defn}
	Given a string $\sigma=(a_1,\dots,a_n)\in S^{<\omega}$ and an ultrafilter $\U\in\beta S$, we define the linear combination of $\U$ times $\sigma$ to be the ultrafilter 
	\begin{equation*}
	    \sigma\U = a_1\U+\dots+ a_n\U.
	\end{equation*}
\end{defn}

Different strings may generate the same ultrafilter: for example, if $\U\in\beta\Z$ is an additively idempotent ultrafilter, then $(1,1)\U=(1)\U$. This leads to introduce the following equivalence relation:

 \begin{defn}
Let $\approx$ be the smallest equivalence relation on $S^{<\omega}$ such that
	\begin{enumerate}
		\item the empty string $()$ is equivalent to $(0)$;
		\item for any $a\in S$, $(a)$ is equivalent to $(a,a)$; and
		\item $\approx$ is coherent with concatenation; i.e. if $\sigma\approx\sigma^{\prime}$ and $\tau\approx\tau^{\prime}$ then $\sigma\tau\approx \sigma^{\prime}\tau^{\prime}$.
	\end{enumerate}

If $\sigma\approx\tau$ we say that $\sigma$ is coherent with $\tau$. The equivalence class of $\sigma\in S^{<\omega}$ under the relation $\approx$ is denoted by $\mathfrak{G}(\sigma)$.
 \end{defn}

As an example, the string $(1,-2,3)\approx (0,1,1,-2,-0,-2,0,0,3,3,0,3)$. It is known that the equivalence relation $\approx$ characterises linear combinations $\sigma\U$ for $\sigma\in\N^{<\omega}$ and $\U\in\beta\Z$, in the following sense:

\begin{thm}\label{theorem:diNasso}
Let $\sigma,\tau\in\mathbb{Z}^{<\omega}$. The following facts are equivalent:
\begin{enumerate}
	\item\label{theorem:diNasso_1} $\sigma\approx\tau$;
	\item\label{theorem:diNasso_2} for every additively idempotent ultrafilter $\U\in\bZ$, $\sigma\U = \tau\U$; and
	\item\label{theorem:diNasso_3} for every additively idempotent ultrafilter $\U\in\bN$, $\sigma\U = \tau\U$.
\end{enumerate}
\end{thm}
\begin{proof}
The proofs of the equivalence (\ref{theorem:diNasso_1})$\Leftrightarrow$(\ref{theorem:diNasso_3}) can be found in \cite[Theorem 3.6]{DiNasso2014}. To prove the equivalence (\ref{theorem:diNasso_1})$\Leftrightarrow$(\ref{theorem:diNasso_2}), first note that any additively idempotent ultrafilter $\U\in\bZ$ satisfies $\U\in\bN\setminus\N$ or $-\U\in\bN\setminus\N$. Hence, the result follows applying \cite[Corollary 4.2]{HindmanLeaderStrauss2003}.
\end{proof}

From the above discussion, it has to be expected that properties of linear combinations $\sigma\U$ should correspond, somehow, with properties of $\sigma$.

We define operations between strings componentwise; when $c\in\Z$ and $\sigma=\left(a_{1},\dots,a_{n}\right)\in S^{<\omega}$, we let $c\sigma:=\left(ca_{1},\dots,ca_{n}\right)$; and, if $\tau=(b_1,\dots,b_n)\in S^{<\omega}$, then $\sigma + \tau := (a_1+b_1,\dots,a_n+b_n)$. Moreover, we let $0$ denote a string with all entries equal to $0$, independently of its length.

\begin{defn}
	A string $\left(a_1,\dots,a_n\right)$ in $\Z^{<\omega}$ is called reduced if
	\begin{enumerate}
		\item for every $i\in\{1,\dots,n-1\}$, $a_i\neq a_{i+1}$; and
		\item for every $i\in\{1,\dots,n\}$, $a_i\neq 0$.
	\end{enumerate}
\end{defn}
\begin{defn}

Let $\sigma\in \Z^{<\omega}$ and let $P\in\Z\left[x_{1},\dots,x_{m}\right]$ be a linear homogeneous polynomial with coefficients $c_1,\dots,c_m$. We say that the equation $P\left(x_1,\dots,x_m\right)=0$ has a solution in $\mathfrak{G}(\sigma)$ if there are strings $\sigma_1,\dots,\sigma_m\in\mathfrak{G}(\sigma)$ all of the same size such that $P\left(\sigma_1,\dots,\sigma_m\right)=\sum_{i=1}^{m}c_i \sigma_i = 0$. This solution will be called injective if the strings $\sigma_1,\dots,\sigma_m$ are mutually distinct.
\end{defn}

Note that, if $\sigma_1,\dots,\sigma_m\in\Z^{<\omega}$ have the same length and are, in this order, the columns of the matrix $M$, then, given $c_1,\dots,c_m\in\Z$, we have that 
\begin{equation*}
    M(c_1,\dots,c_m)^T = c_1\sigma_1+\dots+c_m\sigma_m.
\end{equation*}
Using this fact, one can easily prove the following:
\begin{lem}\label{lemma:string_solution_equiv}
Given $c_1,\dots,c_m\in\Z\setminus\{0\}$, let $P(x_{1},...,x_{m})=\sum\limits_{i=1}^{m}c_{i}x_{i}$ be a homogeneous linear polynomial and $\sigma\in \Z^{<\omega}$. The following conditions are equivalent:
\begin{enumerate}
	\item $P$ has a solution in $\mathfrak{G}(\sigma)$;
	\item there exists a matrix $M$ whose columns are coherent with $\sigma$ such that $M(c_{1},...,c_{m})^T=0$.
\end{enumerate}
Moreover, also the following conditions are equivalent:
\begin{enumerate}
	\item $P$ has an injective solution in $\mathfrak{G}(\sigma)$;
	\item there exists a matrix $M$ whose columns are coherent with $\sigma$ and mutually distinct such that $M(c_{1},...,c_{m})^T=0$.
	\end{enumerate}
\end{lem}

An important fact that we will use about linear combinations of the form $\sigma\U$ is that they contain sets with a peculiar structure. If $n\geq 2$ and $F_1,\dots, F_n$ are finite sets of natural numbers, we write $F_1 <\dots < F_n$ whenever $\max F_i < \min F_{i+1}$, for every $i<n$.

\begin{defn}
Let $(x_n)_{n\in\N}$ be a sequence of integers and $\sigma=(a_1,\dots,a_n)$ be a reduced string of integers. We define the $(\sigma,(x_n)_{n\in\N})$-Milliken-Taylor system $\MT(\sigma,(x_n)_{n\in\N})$ to be the set
 	\begin{equation*}
		\left\{ \sum_{i=1}^{n}a_i \left(\sum_{j\in F_i}x_j\right)\mid  \forall i\leq n \ F_i\in\Pfin(\N) \text{ and } F_1< \dots < F_n  \right\}.
	\end{equation*}
\end{defn}

Special Milliken-Taylor systems will be used in the proof of our main result. Notice that the presence of Milliken-Taylor systems inside sets belonging to some linear combination of an ultrafilter is ensured by the following result\footnote{The result was originally proven for $\U\in\beta\N$, but it is immediate to generalise it to $\beta\Z$ by mapping $\U\in\beta\Z\setminus\beta\N$ to $-\U$.}:

\begin{thm}\cite[Theorem 17.32]{HindmanStrauss2011}\label{theorem:mt-system-linear-com-idem} Let $(x_n)_{n\in\N}$ be a sequence of integers and $\U\in\bZ$ be an additively idempotent ultrafilter such that, for all $m\in\N$, $\FS((x_n)_{n\geq m})\in\U$. Given any reduced string of integers $\sigma$, $\MT(\sigma,(x_n)_{n\in\N})\in\sigma\U$.
\end{thm}

\section{Polynomials solved by linear combinations of ultrafilters}\label{section:Polynomials_solved_by_lin_comb_idemp}

In the previous sections, we recalled (almost) all results that we need to study our main problem, namely: given a reduced string $\sigma\in\Z^{<\omega}$ and an idempotent $\U\in\beta\Z$, for which kind of linear homogeneous polynomials $P\in\Z\left[x_{1},\dots,x_{m}\right]$ does $\sigma\U$ witness the partition regularity of $P\left(x_1,\dots,x_m\right)=0$?

When $\U$ is a minimal idempotent, the short answer is: always. Actually, in this case we have an even more general result:

\begin{thm}
	Let $\U_1,\dots,\U_n\in\bZ$ be additively minimal idempotent ultrafilters, $\left(a_1,\dots,a_n\right)$ a reduced string of integers and $P\in\Z\left[x_1,\dots,x_m\right]$ a Rado polynomial, then $a_1\U_1+\dots+ a_n\U_n$ witnesses the partition regularity of the equation $P\left(x_1,\dots,x_m\right)= 0$
\end{thm}

\begin{proof}
	By the item (iv) of the Theorem \ref{theorem:facts_about_witnessing_pr}, any additively minimal idempotent of $\bN$ witnesses the partition regularity of the equation $P\left(x_1,\dots,x_m\right)=0$; combining this fact with an inductive use of the item (ii) of the Lemma \ref{theorem:facts_about_witnessing_pr}, we conclude our proof.\end{proof}

From the above result we see that for any Rado polynomial $P\in\Z[x_1,\dots,x_m]$ and any additivelly minimal idempotent $\U\in\bZ$, all linear combinations of $\U$ witness the partition regularity of the equation $P\left(x_1,\dots,x_m\right)=0$. It turns out that the additivelly minimal idempotents ultrafilters are not the only class of ultrafilters with this property.

\begin{prop}
	 There exists an additivelly idempotent ultrafilter $\U\in\bZ$ such that
	 \begin{enumerate}
	 	\item $\U$ is not minimal;
	 	\item for all $m\geq 2$, for all Rado polynomial $P\in\Z[x_1,\dots,x_m]$ and for all reduced string $\sigma\in\Z^{<\omega}$, $\sigma\U\models P\left(x_1,\dots,x_m\right)=0$.
	 \end{enumerate}
\end{prop}
\begin{proof}
	From Theorem 2 of \cite{BeiglbockBergelsonDownarowiczFish2009} and item (ii) of the Theorem \ref{theorem:facts_about_witnessing_pr}, we know that any linear combination of an essential idempotent solves any Rado polynomial; in \cite[Theorem 1.14]{BergelsonMcCutcheon2010}, the authors proved that the class of all essential idempotent ultrafilters is strictly larger than the class of all minimal idempotents.
\end{proof}

Now, let ESS be the statement \emph{"there exists a strongly summable ultrafilter on $\Z$"}. Our main result is:

\begin{thm}\label{lingen}(ZFC+ESS) Let $P\in\Z[x_1,\dots,x_m]$ be a linear homogeneous polynomial and $\sigma\in\Z^{<\omega}$ be a reduced string. The following facts are equivalent:
	\begin{enumerate}
		\item\label{lingen1} For all idempotent ultrafilters $\U\in\bZ$ we have that
		\begin{equation*}
			\sigma\U\models P\left(x_1,\dots,x_m\right)= 0;
		\end{equation*}
		\item\label{lingen2} There exists a strongly summable ultrafilter $\U\in\bZ$ such that
		\begin{equation*}
			\sigma\U\models P\left(x_1,\dots,x_m\right) = 0;
		\end{equation*}
		\item\label{lingen3} $P$ has a solution in $\mathfrak{G}(\sigma)$.
	\end{enumerate}
\end{thm}

Although the proof of the equivalence between (\ref{lingen1}) and (\ref{lingen3}) in Theorem \ref{lingen} can be done entirely within ZFC, we first present our proof using strongly summable ultrafilters because they naturally have the string-like structure that we will use in the proof.

The first result we need to recall is the following Theorem\footnote{Di Nasso proved, with nonstandard methods, this result for $\U\in\bN$ and $\sigma\in\N^{<\omega}$; however, its proof generalises in a straightforward way to $\beta\Z$ and $\Z^{<\omega}$.} from \cite{DiNasso2014}, which is a particular case of \cite[Theorem 5.4]{LuperiBaglini2018}, that shows the implication $(3)\Rightarrow (1)$ in Theorem \ref{lingen}.

\begin{lem}\label{lemma:lingen_lemma_1}
Let $\U\in\bZ$ be an additively idempotent ultrafilter, $\sigma\in\Z^{<\omega}$ a reduced string and $P\in\Z[x_1,\dots,x_m]$ be a linear homogeneous polynomial that has a solution in $\mathfrak{G}(\sigma)$. Then $\sigma\U\models P\left(x_1,\dots,x_m\right)=0$.
\end{lem}

The core of the proof of Theorem \ref{lingen} is the following property of strongly summable ultrafilters. Basically, it says that strongly summable ultrafilters contain sets of sums that are so sparse that they behave as strings. 

\begin{lem}\label{lemma:lingen_lemma_2} Let $\U$ be a strongly summable ultrafilter and $\sigma\in\Z^{<\omega}$ be a reduced string. If $P(x_1,\dots,x_m)\in\Z\left[x_{1},\dots,x_{m}\right]$ is a linear homogeneous polynomial such that $\sigma\U\models P(x_1,\dots,x_m)=0$, then $P(x_1,\dots,x_m)=0$ has a solution in $\mathfrak{G}(\sigma)$.
\end{lem}

\begin{proof}
	Let $c_1,\dots,c_m$ be the coefficients of $P$ and consider $\sigma=(a_1,\dots,a_n)$. Let $M = \sum_{i=1}^{n}\left(\sum_{j=1}^{m} |a_i c_j|\right) + 1$. By Lemma \ref{lemma:ss_ultrafilter_main_lemma} there is a sequence of integers $(z_t)_{t\in\N}$ such that, for all $k\in\N$, $\FS\big((z_t)_{t\geq k}\big)\in\U$ and $|z_{t+1}|>M\sum_{i=1}^t |z_i|$. If $Y = \MT(\sigma,(z_t)_{t\in\N})$, by Theorem \ref{theorem:mt-system-linear-com-idem}, we have that $Y\in \sigma\U$. As $\sigma\U\models  P(x_1,\dots,x_m)=0$, one can find $y_1,\dots,y_m\in Y$ such that $P(y_1,\dots,y_m)=0$. By the definition of $Y$, for every $j\leq m$ and $i\leq n$, one can find natural numbers $d_{1,i,j}<\dots< d_{k_{i,j},i,j}$ such that $y_j = \sum_{i=1}^{n}\sum_{t=1}^{k_{i,j}} a_i z_{d_{t,i,j}}$ and, for each $i<n$, $d_{k_j,i,j} < d_{1,i+1,j}$. Hence, $P(y_1,\dots,y_m) = \sum_{j}^{m}\sum_{i=1}^{n}\sum_{t=1}^{k_{i,j}} a_i c_j z_{d_{t,i,j}}$.
	
Consider $\bar{d}=\max\{d_{t,i,j}\mid j\leq m\text{ and } i\leq n\text{ and } t\leq k_{i,j}\}$ and for each $i\leq n$, $j\leq m$ and $d\in\{0,\dots, \bar{d}\}$, define
\begin{equation*}
	\delta_{i,j,d} = \begin{cases}
		1, &\text{ if } a_i z_d \text{ appears in } y_j;\\
		0, & \text{ otherwise.}
	\end{cases}
\end{equation*}
Then,
\begin{equation}\label{eq:p_ys}
	P(y_1,\dots,y_m) = \sum_{i=1}^{n}\sum_{j=1}^{m}\sum_{d=0}^{\bar{d}} a_i c_j \delta_{i,j,d} z_d.
\end{equation}

We claim that the following are equivalent:
\begin{description}[labelindent=1cm]
	\item[a)] $P(y_1,\dots,y_m) = 0$;
	\item[b)] for all $d\in\{0,\dots,\bar{d}\}$, $\sum_{i=1}^{n}\sum_{j=1}^{m}a_ic_j\delta_{i,j,d}=0$.
\end{description}

That b) implies a) is immediate. To prove that a) implies b), we proceed by contradiction. If b) is false, it is possible to find the greatest element $\tilde{d}$ among all  $d\in\{0,\dots,\bar{d}\}$ such that $\sum_{i=1}^{n}\sum_{j=1}^{m}a_ic_j\delta_{i,j,d}\neq 0$. By definition, for any $d > \tilde{d}$, $\sum_{i=1}^{n}\sum_{j=1}^{m}a_ic_j\delta_{i,j,d}=0$. Hence, from the equation (\ref{eq:p_ys}), one has
\begin{equation}\label{eq:2}
	\sum_{i=1}^{n}\sum_{j=1}^{m}\sum_{d=0}^{\tilde{d}} a_i c_j \delta_{i,j,d} z_d = 0.
\end{equation}
Then,
\begin{equation*}
	\left| \sum_{d=0}^{\tilde{d}-1}\sum_{i=1}^{n}\sum_{j=1}^{m} a_i c_j \delta_{i,j,d} z_d \right| \leq \sum_{d=0}^{\tilde{d}-1} \left(\sum_{i=1}^{m}\sum_{j=1}^{m} |a_i c_j|\right) z_d < M\left|\sum_{d=0}^{\tilde{d}-1} z_d\right| < z_{\tilde{d}},
\end{equation*}
which is absurd since, from equation (\ref{eq:2}) one can derive
\begin{equation*}
	\left| \sum_{d=0}^{\tilde{d}-1}\sum_{i=1}^{n}\sum_{j=1}^{m} a_i c_j \delta_{i,j,d} z_d \right| = \left| \sum_{i=1}^{n}\sum_{j=1}^{m} a_i c_j \delta_{i,j,\tilde{d}} z_{\tilde{d}} \right|.
\end{equation*}

Hence, the claim is proved.

Let us define an $m\times (\bar{d}+1)$ $S$ by constructing its columns. For each $j\leq m$, the $j$-th column $s_j = (s_{d,j})_{d=0}^{\bar{d}}$ of $S$ is defined as
\begin{equation*}
	s_{d,j} = \begin{cases}
		a_j, & \text{ if } \delta_{i,j,d} = 1;\\
		0, & \text{ otherwise.}
	\end{cases}
\end{equation*}

This vector is well defined as for all $j_{1},j_{2}\leq m$, $i\leq n$ and $d\in\{0,\dots,\bar{d}\}$, we have that $\delta_{i,j_1,d}=\delta_{i,j_2,d}=1$ implies $j_1=j_2$. As each $y_{j}$ has the form $y_{j}=a_{1}z_{d_{1,1,j}}+\dots+a_{1}z_{d_{k_{i,j},1,j}}+a_{2}z_{d_{1,2,j}}+\dots+a_{n}z_{d_{k_{n,j},n,j}}$ with $d_{t,i,j}< d_{t+1,i,j}$ for all $j\leq m$, $i\leq n$  and $t< k_{i,j}$, and $d_{k_{i,j},i,j}< d_{1,i+1,j}$ for each $i<n$, we trivially have that each $s_j$ is equivalent to $\sigma$; thus $S$ is coherent with $\sigma$. Therefore, for each $d\in \{0,\dots,\bar{d}\}$, the $d$-th component of $S(c_1,\dots, c_m)^{T}$ is $\sum_{i=1}^{m}\sum_{j=1}^{n}c_i a_j \delta_{i,j,d}$, which is zero by the claim b). Invoking Lemma \ref{lemma:string_solution_equiv}, we have that the equation $P\left(x_{1},\dots,x_{m}\right)=0$ has a solution in $\mathfrak{G}(\sigma)$.
\end{proof}

\section{Foundational Issues}\label{FI}

In the formulation and proof of Theorem \ref{lingen} we used strongly summable ultrafilters for they contain sets of sums that behave like strings, a property that is not shared by all ultrafilters. However, the proof of the equivalence between (1) and (3) in Theorem \ref{lingen} can be done entirely in ZFC; to this end, we need to find additively idempotent ultrafilters that can play the role of strongly summable ultrafilters in the proof of Lemma \ref{lemma:lingen_lemma_2}. The following known fact will enable us to find such ultrafilters:

\begin{lem}\cite[Lemma 5.11]{HindmanStrauss2011}\label{lemma:sigma_p_ultra}
	Consider $S$ to be $\N$ or $\Z$. Let $\left(x_t\right)_{t\in\N}$ be a sequence of elements of $S$. Then, there is an idempotent $\U\in\beta S$ such that, for every $k\in\N$, $\FS\big((x_t)_{t\geq k}\big)\in \U$.
\end{lem}

\begin{defn}
Giving $c_1,\dots,c_m\in\Z\setminus\{0\}$, let $P\left(x_1,\dots,x_m\right)=\sum_{i=1}^{m}c_{i}x_{i}$ and $\sigma=(a_1,\dots,a_n)\in\Z^{<\omega}$ be a reduced string. Define $M = \sum_{i=1}^n\left(\sum_{j=1}^m |a_i c_j| \right) + 1$. We say that an idempotent ultrafilter $\U\in\bZ$ is a $(\sigma,P)$-ultrafilter if there is a sequence $(x_t)_{t\in\N}$ in $\Z$ such that, for each $k\in\N$, $\FS\big((x_t)_{t\geq k}\big)\in\U$ and $|x_{k+1}|>M|\sum_{i=1}^k x_i|$.
\end{defn}

Lemma \ref{lemma:sigma_p_ultra} easily implies the existence of $(\sigma,P)$-ultrafilters in ZFC. Using the existence of $(\sigma,P)$-ultrafilters and repeating, mutatis mutandis, the arguments of the proof of the Lemma \ref{lemma:lingen_lemma_2}, one can settle the following analogue of Theorem \ref{lingen} in ZFC:

\begin{thm}
	Let $\sigma\in\Z^{\omega}$ be a reduced string and $P(x_1,\dots,x_m)\in\Z\left[x_{1},\dots,x_{m}\right]$ be linear and homogeneous. Then, the following are equivalent:
	\begin{enumerate}
		\item for all idempotents $\U\in\bZ$, $	\sigma\U\models P\left(x_1,\dots,x_m\right) = 0$;
		\item there exists a $(\sigma,P)$-ultrafilter $\U\in\bZ$ such that $\sigma\U\models P\left(x_1,\dots,x_m\right) = 0$;
		\item $P$ has a solution in $\mathfrak{G}(\sigma)$.
	\end{enumerate}
\end{thm}

Let us observe that the conditions defining a $(\sigma,P)$-ultrafilter are much weaker than those defining a strongly summable ultrafilter, since we are not asking that every $A\in\U$ must contain the finite sums of some sequence, but just that there is one such set with the needed growth condition in $\U$. In fact, a strongly summable ultrafilter is automatically a $(\sigma,P)$-ultrafilter for all $\sigma\in\Z^{<\omega},P$ linear homogeneous polynomial. It is unclear (to us) if the converse holds as well:

\begin{question} Is it provable in ZFC that there exists $\U\in\beta\Z$ that is a $(\sigma,P)$-ultrafilter for all $\sigma\in\Z^{<\omega}$ and all linear homogeneous polynomial $P$? If not,  is it true that such an ultrafilter is necessarily a strongly summable ultrafilter?
\end{question}

Moreover, we know that if $\U$ is an essential ultrafilter, then for any $\sigma\in\Z^{<\omega}$ $\sigma\U$ solves all Rado equations. We do not know if this implication can be reversed\footnote{We know that there are ultrafilters that are not idempotent with this property, for example any $\U\in\overline{K(\beta\Z,\odot)}$.}, so we conclude this section with one last question:

\begin{question}
Let $\U$ be an additively idempotent ultrafilter of $\beta\Z$ such that for each $\sigma\in\Z^{<\omega}$, $\sigma\U$ witnesses the partition regularity of any Rado equation. Is it true that $\U$ must be essential? I.e. among all additively idempotent ultrafilters of $\beta\Z$, is the class of essential idempotent ultrafilters maximal with the respect of the property of witnessing the partition regularity of all Rado equations?
\end{question}

\section{Examples}\label{exampleeee}

\subsection{The case $\sigma=(1)$}\label{s1}

\phantom{s}

In this subsection, we study which linear homogeneous equations with integers coefficients are witnessed by all idempotents $\U\in\beta\Z$ and hence, by Theorem \ref{theorem:facts_about_witnessing_pr}, by all possible linear combinations of $\sigma\U$ with integer coefficients. From Theorem \ref{lingen}, one can see that this is equivalent to ask when such equations have a solution in $\mathfrak{G}\big((1)\big)$.

\begin{lem}\label{c=1}
	Giving $c_1,\dots,c_m\in\Z\setminus\{0\}$, the equation $c_1 x_1+\dots+c_m x_m=0$ has a solution in $\mathfrak{G}\big((1)\big)$ if and only if for every $j\leq m$, there is a non-empty $H_j\subseteq \{1,\dots,m\}\setminus\{j\}$ such that $c_j + \sum_{l\in H_j}c_l = 0$.
\end{lem}
\begin{proof} Let $\mathbf{c}=\left(c_{1},\dots,c_{m}\right)$. If the equation has a solution in $\mathfrak{G}\big((1)\big)$, then there is a $k\times m$ matrix $M=(\alpha_{ij})$ coherent with $(1)$ such that $M\mathbf{c}^T=0$. This means that, for each $j\leq m$ there will be a $i\leq k$ such that $1$ appears in the $j$th position of the vector $\alpha_i = (\alpha_{i1},\dots,\alpha_{im})$. Define $H_j = \{l\leq m\mid l\neq j \text{ and } \alpha_{il}\neq 0\}$, then $H_j=\{l\leq m\mid l\neq j \text{ and } \alpha_{il} = 1\}$. As $c_j\neq 0$ and $M\mathbf{c}^T = 0$, $H_j\neq\emptyset$ and $c_j + \sum_{l\in H_j}c_l = 0$.

Conversely, for each $j,l\leq m$, define
	\begin{equation*}
		\alpha_{jl} =
		\begin{cases}
			1, & \text{ if } l\in H_j\cup \{j\}; \text{ or }\\
			0, & \text{ otherwise }
		\end{cases}
	\end{equation*}
and let $M=(\alpha_{jl})$. Note that, for each $j\leq m$, $\alpha_{jj}=1$; hence $M$ is coherent with $(1)$. Moreover, for each $j\leq m$, the $j$th coordinate of $M\mathbf{c}^T$ is $\sum_{l=1}^{m}\alpha_{jl}c_l = c_j + \sum_{l\in H_j} c_l = 0$, which proves that $M\bold{c}^T=0$.
\end{proof}

\begin{exam} The partition regularity of the equation $4x_{1}+2x_{2}+3x_{3}-5x_{4}-x_{5}-2x_{6}=0$ is witnessed by all idempotent ultrafilters $\U\in\beta\Z$, as its coefficients satisfy the condition of Proposition \ref{c=1}. The same holds for the equation $2x_{1}-2x_{2}-x_{3}-x_{4}=0$.
By Lemma \ref{lemma:string_solution_equiv}, we have that the equation $4x_{1}+2x_{2}+3x_{3}-5x_{4}-x_{5}-2x_{6}=0$ admits injective solutions in $\mathfrak{G}((1))$, as

\begin{equation*}
		M=\begin{pmatrix}
			1 & 1 & 0 & 1 & 1 & 0 \\
			0 & 1 & 1 & 1 & 0 & 0 \\
			0 & 0 & 1 & 0 & 1 & 1 \\
			0 & 1 & 0 & 0 & 0 & 1
		\end{pmatrix}\cdot
		\begin{pmatrix}
		4\\
		2\\
		3\\
		-5\\
		-1\\
		-2
		\end{pmatrix}=
		\begin{pmatrix}
		0\\
		0\\
		0\\
		0
		\end{pmatrix}
	\end{equation*}
and the columns of $M$ are mutually distinct; however, the equation $2x_{1}-2x_{2}-x_{3}-x_{4}=0$ does not admit injective solutions in $\mathfrak{G}((1))$: in fact, if $\sigma_{1},\sigma_{2},\sigma_{3},\sigma_{4}\in\mathfrak{G}((1))$ are strings of length $k$ such that $2\sigma_{1}-2\sigma_{2}-\sigma_{3}-\sigma_{4}=0$, for every $i\leq k$, we have that $\sigma_{3,i}=1$ if and only if $\sigma_{4,i}=1$ (where $\sigma_{h,i}$ denotes the $i-th$ entry in the string $\sigma_{h}$). In fact, it is immediate to see that if $\sigma_{3,i}=1$ then necessarily $\sigma_{1,i}=1,\sigma_{2,i}=0,\sigma_{4,i}=1$, and the same if we let $\sigma_{4,i}=1$. Therefore necessarily $\sigma_{3}=\sigma_{4}$.
\end{exam}
\subsection{The case $P(x_{1},x_{2},x_{3})=c_{1}x_{1}+c_{2}x_{2}+c_{3}x_{3}$}

\phantom{a}

A characterisation similar to that of Proposition \ref{c=1} could be given, in principle, for strings of arbitrary length, but at the cost of readability. Here we reverse the problem, and we ask for which strings $\sigma=\left(a_{1},\dots,a_{n}\right)$ can we solve the equation $P\left(x_1,x_2,x_3\right) = c_1x_1+c_2x_2+c_3x_3=0$ in $\mathfrak{G}(\sigma)$. We know that in such a case $P\left(x_{1},x_{2},x_{3}\right)$ must be a Rado polynomial, i.e. some non-empty subset of $\{c_1,c_2,c_3\}$ sums zero.

 We will divide the treatment into two parts: first the case in which $c_1+c_2+c_3=0$ and, second, the case in which there is a pair inside $\{c_1,c_2,c_3\}$ that sums to zero.\\

\textbf{Case 1:} $c_1+c_2+c_3=0$. If we allow for constant solutions, trivially this equation will have a solution in $\mathfrak{G}(\sigma)$ for all possible choices of $\sigma$. Therefore we restrict here to study the conditions under which such polynomials have an injective solution in $\mathfrak{G}(\sigma)$. Without loss of generality, we can assume $\gcd(c_2,c_3)=1$.

\begin{oss}\label{obs} Let $\sigma\in\Z^{<\omega}$. Then $P$ has an injective solution in $\mathfrak{G}((\sigma))$ if and only if it has a non-constant solution in $\mathfrak{G}((\sigma))$. 
\end{oss}

\if
In fact, the implication is immediate, as injective solutions are trivially non constant; for the reverse implication, if $\sigma_{1},\sigma_{2},\sigma_{3}\in\mathfrak{G}((\sigma))$ are not mutually distinct and solve $c_{1}x_{1}+c_{2}x_{2}+c_{3}x_{3}=0$, then at least two of them (say, $\sigma_{1}$ and $\sigma_{2}$) are equal. But then, as $c_{1}+c_{2}=-c_{3}$, we deduce that $-c_{3}\sigma_{1}+c_{3}\sigma_{3}=0$ which forces $\sigma_{1}=\sigma_{3}$, i.e. a constant solution.
\fi 

To solve our problem we first notice that, without loss of generality, we can restrict to the case $\sigma=\left(a_{1},a_{2}\right)$. In fact, we have the following:

\begin{lem}\label{thm:injective_sol_pol_3_var_sum_coef_zero}
The equation $P\left(x_{1},x_{2},x_{3}\right)=0$ has an injective solution in $\mathfrak{G}(\sigma)$ if and only if there is a $i<n$ such that it has a injective solution in $\mathfrak{G}\left(\left(a_i,a_{i+1}\right)\right)$.
\end{lem}
\begin{proof} Suppose that the equation has a injective solution in $\mathfrak{G}(\sigma)$. Then, by Lemma \ref{lemma:string_solution_equiv}, there is a matrix $M$ whose columns are coherent with $\sigma$ and all pairwise distinct and satisfies $M\left(c_{1},\dots,c_{m}\right)^{T}=0$. Note that, since $M$ is coherent with $\sigma$ and $c_1+c_2+c_{3}=0$, the first line of $M$ must be $(a_1,a_1,a_1)$; as a consequence, since the solution is injective, for at least one index $1<l\leq n$ there is a row $L$ in $M$ which is the first non constant where $a_{l}$ appears, say $L=\left(h_{1},h_{2},h_{3}\right)$. Again, the condition $c_{1}+c_{2}+c_{3}=0$ forces $h_{1},h_{2},h_{3}$ to be pairwise different. This means that the previous line should be either $L_1=(a_{l-1},a_{l-1},a_{l-1})$ or $L_2 = (a_l,a_l,a_l)$.  If the previous line is $L_1$, as $M$ is coherent with $\sigma$, $L$ must be a permutation of $(0,a_l,a_{l-1})$; in this case, let 
\begin{equation*}
    M_1 = 
    \begin{pmatrix}
        a_{l-1} & a_{l-1} & a_{l-1}\\
        h_1 & h_2 & h_3\\
        a_l & a_l & a_l
    \end{pmatrix}
\end{equation*}
We have that $M_1$ is coherent with $(a_{l-1},a_l)$ and that $M_1(c_1,c_2,c_3)^T=0$, meaning that $P$ has a solution in $\mathfrak{G}\left(a_{l-1},a_{l}\right)$. Otherwise, the previous line should be $L_2$ and, for an analogous reason, $L$ must be a permutation of $(0,a_{l},a_{l+1})$; in this case, let 
\begin{equation*}
    M_2 = 
    \begin{pmatrix}
        a_l & a_l & a_l\\
        h_1 & h_2 & h_3\\
        a_{l+1} & a_{l+1} & a_{l+1}
    \end{pmatrix}
\end{equation*}
We have that $M_2$ is coherent with $(a_{l},a_{l+1})$ and that $M_2(c_1,c_2,c_3)^T=0$, meaning that $P$ has a solution in $\mathfrak{G}\left(a_{l},a_{l+1}\right)$. In both cases, the columns of the matrices are non-constant which, by Observation \ref{obs}, guarantees the presence of an injective solution.

Conversely, if the equation has a injective solution in $\mathfrak{G}\left(\left(a_i,a_{i+1}\right)\right)$, for some $i<n$, then there are distinct $h_1,h_2,h_3\in\{0,a_i,a_{i+1}\}$ such that $P(h_1,h_2,h_3)=0$. For each $j\leq 3$, let $\sigma_j =(a_1,\dots,a_i,h_j,a_{i+1},\dots,a_n)$. Then, each $\sigma_j$ is coherent with $\sigma$ and $P(\sigma_1,\sigma_2,\sigma_3)=0$. By Lemma \ref{lemma:string_solution_equiv}, the equation has a solution in $\mathfrak{G}(\sigma)$.
\end{proof}

The set of solutions of the equation $P\left(x_{1},x_{2},x_{3}\right)=0$ is a linear space generated by the vectors $(1,1,1)$ and $\left(0,c_{3},-c_{2}\right)$, so all solutions of the equation have the form
  \begin{equation*}
  x_1 = t,\;\;\; x_2 = t + c_3u\;\;\;\text{and}\;\;\; x_3 = t - c_2u
  \end{equation*}
for some $t,u\in\mathbb{Z}$. This simple observation allows to easily deduce the following:

\begin{lem}\label{facile} If $c_{1}+c_{2}+c_{3}=0$ and $\gcd\left(c_{1},c_{2},c_{3}\right)=\gcd\left(a_{1},a_{2}\right)=1$, the equation $P\left(x_{1},x_{2},x_{3}\right)=0$ has an injective solution in $\mathfrak{G}\left(\left(a_{1},a_{2}\right)\right)$ if and only if there are $1\leq i,j\leq 3$ such that a permutation of $\left(a_1,a_{2}\right)$ is equal to $\left(c_{i},-c_{j}\right)$.
\end{lem}

\begin{proof} There is an injective solution in $\mathfrak{G}\left(\left(a_{1},a_{2}\right)\right)$ if and only if there are $t,u\in\mathbb{Z}$ such that $(t,t+c_{3}u,t-c_{2}u)$ is a permutation of $(0,a_{1},a_{2})$ . Now we just have to check the possible cases. 

If $t-c_2u=0$, then either $a_1=t=c_2u$ and $a_2=t+c_{3}u=(c_2+c_3)u$ or (analogously) $a_1=(c_2+c_3)u$ and $a_2=c_2u$; in both cases, as $\gcd(a_1,a_2)=1$, we have $u=1$; from the fact that $c_1+c_2+c_3=0$, we have in the first case $a_1=c_2$ and $a_2=-c_1$, and in the second case $a_1=-c_1$ and $a_2=c_2$. The cases $t+c_3u=0$ or $t=0$ are analogous.

\end{proof}

Putting together Lemmas \ref{thm:injective_sol_pol_3_var_sum_coef_zero} and \ref{facile}, and observing by linearity that $c_{1}x_{1}+c_{2}x_{2}+c_{3}x_{3}=0$ is solvable in $\mathfrak{G}\left(\sigma\right)$ if and only if $kc_{1}x_{1}+kc_{2}x_{2}+kc_{3}x_{3}=0$ is solvable in $\mathfrak{G}\left(h\sigma\right)$ for some $h,k\in\Z\setminus\{0\}$, we deduce the general characterisation.

\begin{prop} If $c_{1}+c_{2}+c_{3}=0$ and $\sigma=\left(a_{1},\dots,a_{n}\right)\in\Z^{<\omega}$, the equation $P\left(x_{1},x_{2},x_{3}\right)=0$ has an injective solution in $\mathfrak{G}(\sigma)$ if and only if there are $1\leq i,j\leq 3$, $h\leq n-1$ and $r,s\in\Z\setminus\{0\}$ such that $r\left(a_h,a_{h+1}\right)=s\left(c_{i},-c_{j}\right)$.
\end{prop}

\begin{exam}Let $P\left(x_1,x_2,x_3\right) = 3x_1 - 5x_2 + 2x_3$ and consider $\sigma = (5,7,-10,-6,13)$. Then $P\left(x_1,x_2,x_3\right)=0$ has an injective solution in $\mathfrak{G}(\sigma)$, because $(-10,-6)$ is multiple of $(-5,-3)$.
\end{exam}

\textbf{Case 2:} Let us suppose now that there is a pair among $c_1$, $c_2$ and $c_3$ that sums zero; without loss of generality, we can consider $c_1=-c_2 = c$,  $c_3=d$ and $\gcd(c,d)=1$, namely $P\left(x_1,x_2,x_3\right)= c\left(x_1-x_2\right) + dx_3$. We shall show that the only equation satisfying the mentioned conditions solvable in $\mathfrak{G}(\sigma)$, for some $\sigma\in\Z^{<\omega}$, is Schur's equation $x_{1}-x_{2}+x_{3}=0$.

\begin{prop}\label{ultimo} Consider the polynomial $P\left(x_1,x_2,x_3\right)= c\left(x_1-x_2\right) + dx_3$, with $c,d\in\Z\setminus\{0\}$ and $\gcd(c,d)=1$. The following facts are equivalent:
\begin{enumerate}
\item there is a non-empty reduced string $\sigma=\left(a_{1},\dots,a_{n}\right)$ such that $\sigma\U\models P\left(x_1,x_2,x_3\right)=0$ for all additively idempotent ultrafilters $\U\in\beta\Z$;
\item $\sigma\U\models P\left(x_1,x_2,x_3\right)=0$ for all additively idempotent ultrafilters $\U\in\beta\Z$ and for all non empty reduced strings $\sigma\in\Z^{<\omega}$;
\item $|c|=|d|=1$, i.e. $P\left(x_1,x_2,x_3\right)$ is Schur's polynomial $x_{1}+x_{2}-x_{3}$.
\end{enumerate}
\end{prop}

\begin{proof} The implication $(3)\Rightarrow (2)$ is a consequence of Theorem \ref{theorem:facts_about_witnessing_pr} and the fact that any member of an additively idempotent ultrafilter on $\Z$ is an IP-set; the implication $(2)\Rightarrow (1)$ is trivial.

To conclude, let us show that $(1)\Rightarrow (3)$. Let $\sigma=\left(a_{1},\dots,a_{n}\right)$, where we assume without loss of generality that $\gcd(a_1,\dots,a_n)=1$, and assume that the equation $P\left(x_1,x_2,x_3\right)=0$ has a solution in $\mathfrak{G}(\sigma)$. First, let us show that it must be that $|c|=1$. In fact, let $M$ be a $k\times 3$, coherent matrix with $\sigma$, such that $M(c,-c,d)=0$; as each $a_i$ has to appear in the third column of $M$, we always find  $u_i,v_i\in\{0,a_1,\dots,a_n\}$ such that
\begin{equation}\label{equation:ultimo_resultado}
c\left(u_i-v_i \right)=da_{i} 
\end{equation}

As $\gcd(c,d)=1$, the equation (\ref{equation:ultimo_resultado}) shows that, for each $i\leq n$, $c$ divides $a_{i}$. As we assumed $gcd\left(a_{1},\dots,a_{n}\right)=1$, this forces $|c|=1$.

Now we prove that $|d|\geq 3$ cannot occur. Indeed, suppose that $|d|\geq 3$ and consider $a=\max\{|a_{i}| \mid i\leq n\}$. Let $i_0:=\min\{i\leq n\mid  |a_{i}|=a\}$. By changing $\sigma$ with $-\sigma$, if necessary, we can assume without loss of generality that $a_{i_{0}}=a$. By looking at the appearance of $a_{i_0}$ in the third column of $M$, one can find $u_0,v_0\in\{0,a_1,\dots,a_n\}$ such that $u_0 - v_0 + da_{i_0}=0$. Hence, we have that $|u_0 - v_0|=|da|$, which is absurd, since $|u_0 - v_0|\leq 2a$.

Finally, let us prove that $|d|$ cannot be 2. Let us proceed by contradiction; by multiplying $P\left(x_{1},x_{2},x_{3}\right)$ by $-1$, if needed, we can assume that $d=2$. The idea is to show that, in the string $\sigma=\left(a_{1},\dots,a_{n}\right)$, whenever $a$ appears with index $i\leq n$, then $-a$ must appear with an index $j>i$; and, analogously, whenever $-a$ appear with an index $j'$, then $a$ must appear with an index $i'>j'$. Of course this cannot happen, as it would generate an infinite substring $(a,-a,a,-a,\dots)$ of the finite string $\sigma$. To construct this substring, we proceed by recursion on the indices of $\sigma$. Let again $a = \max\{|a_i|\mid i\leq n\}$ and $i_0 = \min\{i\leq n\mid |a_i|=a\}$.

As before, we may assume that $a_{i_0}=a$. Then, there are $u_0,v_0\in\{0,a_1,\dots,a_n\}$ such that $(u_0, v_0,a)$ is a line of $M$, say with index $k_0\leq k$, which implies that $u_0 - v_0=2a$. Hence, by the maximality of $a$, it must be $u_0=a$ and $v_0=-a$. As $v_0\neq 0$ and by the minimality of $i_0$, 
\begin{equation*}
    I_0 = \{i\leq n\mid  i>i_0\text{ and } a_i = -a\}\neq\emptyset. 
\end{equation*}
Let $i_1=\min I_0$ and, by minimality of $i_0$, observe that $a_{i_0}$ cannot appear in the second column of any line of $M$ whose index is greater than $k_0$.

Now assume that, for a $r\leq n$, we have defined two increasing lists, namely $i_0,\dots,i_r\in\{1,\dots,n\}$ and $k_0,\dots,k_r\in\{1,\dots,k\}$, indexing the components of $\sigma$ and the lines of the matrix $M$, respectively, in such a way that
\begin{enumerate}
  \item for each $j\in\{0,\dots,r\}$, $a_{i_j}=(-1)^ja$;
  \item for each $j\in\{0,\dots,r\}$, $a_{i_j}$ appears in the third column of $M$ at the line $k_j$; and
  \item fixing $j,l\in\{0,\dots,r\}$ such that $l\leq j$, $a_{i_l}$ do not appear in the second column of any line greater than or equal to $k_j$.
\end{enumerate}

Thus, there must be $u_{r},v_{r}\in\{0,a_1,\dots,a_n\}$ such that the $k_r$-th line of $M$ is $(u_{r},v_{r},a_{i_r})$, which means that $u_r-v_r=2a_{i_r}$, which can only happen if $u_r=a_{i_r}$ and $v_r=-a_{i_r}$. By the item (3) above, we see that 
\begin{equation*}
   I_r = \{i\leq n\mid  i>i_r\text{ and }a_{i}=-a_{i_r}\}\neq\emptyset.  
\end{equation*}

Define $i_{r+1}=\min I_r$ and let $k_{r+1}$ be the line of $M$ where $a_{i_{r+1}}$ appears in the third column. Then, the lists $i_0,\dots,i_{r+1}$ and $k_0,\dots,k_{r+1}$ must be increasing and also satisfy conditions (1), (2) and (3) above. This shows that if $d=2$, there is an infinite substring $(a,-a,a,-a,\dots)$ of the finite string $\sigma$, which is absurd.
\end{proof}

Notice that, as a straight consequence of Proposition \ref{ultimo}, we get the existence of plenty of Rado equations that are not solvable in $\mathfrak{G}(\sigma)$, for all $\sigma\in\Z^{<\omega}$. Precisely:
\begin{cor}\label{shursequation}
Let $c,d\in\Z\setminus\{0\}$ such that $\gcd(c,d)=1$. Then, the following are equivalent:
\begin{enumerate}
    \item there exists a $\sigma\in\Z^{<\omega}$ such that, for all additively idempotent ultrafilter $\U\in\bZ$, $\sigma\U\models c(x-y)+dz=0$; and
    \item c=d=1.
\end{enumerate}
\end{cor}

\section{Acknowledgements}
The authors thank Sohail Farhangi for pointing us to the proof of the equivalences between (\ref{theorem:diNasso_1}) and (\ref{theorem:diNasso_2}) in the Theorem \ref{theorem:diNasso}. The authors also thank the anonymous reviewer for her/his helpful comments.

\end{document}